\documentclass[final]{siamltex}
\usepackage{amscd,amssymb,epsfig,mathtools}
\usepackage{times}
\usepackage[letterpaper=true,colorlinks=true,
        linkcolor=red,filecolor=green,citecolor=red,
        pdfpagemode=None]{hyperref}

\newcommand{\diam}{\operatorname{diam}}

\newcommand{\R}{\mathbb{R}}

\newcommand\Conv{\operatorname{Conv}}

\newcommand{\tir}[1]{\ensuremath{\overline {#1}}} 
\newtheorem{thm}{Theorem}[section]

\newtheorem{ass}[thm]{Assumption}

\newtheorem{defn}[thm]{Definition} 
 
\newtheorem{rem}[thm]{Remark}

\def\whsq{\vbox to 5.8pt 
{\offinterlineskip\hrule 
\hbox to 5.8pt{\vrule height 
5.1pt\hss\vrule height 5.1pt}\hrule}}

\def\<{\langle} 
\def\>{\rangle} 
\def\PP{{\mathop{{\rm I}\kern-.2em{\rm P}}\nolimits}} 
\def\FF{{\mathop{{\rm I}\kern-.2em{\rm F}}\nolimits}}   
\def\ZZ{{\mathop{{\rm I}\kern-.2em{\rm Z}}\nolimits}}

\title{Uniform limit of discrete convex functions }

\author{Gerard Awanou\thanks{Department of Mathematics, Statistics, and Computer Science, M/C 249.
University of Illinois at Chicago, 
Chicago, IL 60607-7045, USA ({\tt awanou@uic.edu}).}
        }

\begin{document}

\maketitle

\begin{abstract}
We consider mesh functions which are discrete convex in the sense that their central second order directional derivatives are positive. Analogous to the case of a uniformly bounded  sequence of convex functions, we prove that the uniform limit on compact subsets of discrete convex mesh functions which are uniformly bounded is a continuous convex function. Furthermore, if the discrete convex mesh functions
 interpolate boundary data of a continuous convex function and their Monge-Amp\`ere masses are uniformly bounded, the limit function satisfies the boundary condition strongly. The domain of the solution needs not be  uniformly convex. The result is applied to the convergence of some numerical methods for the Monge-Amp\`ere equation.

\end{abstract}



\pagestyle{myheadings}
\thispagestyle{plain}
\markboth{G. AWANOU }{Uniform limit of discrete convex functions}

\section{Introduction}

Let $\Omega$ be a  bounded convex domain of $\R^d$ with boundary $\partial \Omega$ and let $g \in C(\partial \Omega)$. We assume that $g$ can be extended to a  convex function $\tilde{g} \in C(\tir{\Omega})$. In this paper we
prove that the uniform limit on compact subsets of  mesh functions which are discrete convex, c.f. Definition \ref{discrete-convex} below, and which interpolate $g$  on $\partial \Omega$, is a continuous convex function on $\tir{\Omega}$ which solves $v=g$ on $\partial \Omega$, provided their Monge-Amp\`ere masses, c.f. Definition \ref{mass} below, are uniformly bounded. The corresponding result for sequences of convex functions can be found in \cite[Lemma 5.1]{Hartenstine2006}.

The standard arguments for convergence of schemes to viscosity solutions of elliptic equations, which are based on consistency, stability and monotonicity, require the equation to satisfy a comparison principle for Dirichlet boundary condtions in the sense of viscosity \cite{NeilanSalgadoZhang2017,daniel2014quelques}. A comparison principle is only known to hold for the Monge-Amp\`ere equation when the Dirichlet boundary condition holds strongly. 
Nethertheless convergence results have been proved, but it has been required that the domain be smooth and uniformly convex 
\cite{Feng2017,jensen2017numerical,nochetto2018two}.  Here we do not assume that the domain is smooth, nor do we assume that it is strictly convex. As an application of our result, if it is known that the discrete solutions  are uniformly bounded, discrete convex and with Monge-Amp\`ere masses uniformly bounded, it can be shown that a
subsequence converges uniformly on compact subsets to a convex function continuous up to the boundary and which solves the boundary condition 
strongly. One then only needs to prove that the uniform limit obtained is a viscosity solution of the equation. This is done the usual way, based on the consistency and monotonicity of the discretization, by showing that the uniform limit is both a viscosity sub solution and a viscosity super solution of the differential equation. 
The uniform limit is then a viscosity solution of the Dirichlet problem and hence unique by the known comparison principle, c.f. Theorem \ref{comparison} below. Thus all subsequences converge to the same limit, proving the convergence of the discretization.

Our contribution offers a novel tool for the study of the convergence of some existing discretizations, e.g. \cite{Oberman2010a,Mirebeau14,Mirebeau15,nochetto2018two}. The main ingredients of our approach are the discrete Aleksandrov-Bakelman-Pucci's maximum principle
and the discrete maximum principle for the discrete Laplacian. 
The assumption on the uniform bound on the Monge-Amp\`ere masses is easily checked when the right hand side is integrable for geometric methods such as \cite{Mirebeau14,Mirebeau15,DiscreteAlex2}. 
The axiomatic approach for convergence of finite difference schemes to the Aleksandrov solution of the Monge-Amp\`ere equation presented in \cite{AwanouAwi22} uses  the main result proved in this paper. We consider only uniform meshes in this paper.

The analysis of numerical methods for the Monge-Amp\`ere equation is an active research area. The references \cite{Baginski96,Belgacem2006,GlowinskiICIAM07,Delzanno08,Bohmer2008,Feng2009,Chen2010b,Zheligovsky10,Brenner2010b,Lakkis11b,Sulman11,Davydov12,Glowinski2014,FengLewis14,Brix,neilan2019monge} cover most of the various approaches. 

The paper is organized as follows. In the next section we collect some notation used throughout the paper and prove 
in section \ref{result}
the theorem which asserts that the uniform limit on compact subsets of a sequence of uniformly bounded discrete convex functions interpolating the boundary condition is a continuous convex function which satisfies it, provided their Monge-Amp\`ere masses are uniformly bounded. In section \ref{cvg}, we review the notion of viscosity solution, monotone schemes and detail our proof of convergence strategy for monotone discretizations of the Monge-Amp\`ere equation. The proof that the uniform limit is convex is stated as a lemma in section \ref{result} and proved in the last section, as it relies on the notion of viscosity solution and monotone schemes which are reviewed in section  \ref{cvg}.

\section{Preliminaries}  \label{as}

We use the notation $|| . ||$ for the Euclidean norm of $\R^d$. Let $h$ be a small positive parameter and let
$$
\mathbb{Z}^d_h = \{\, m h, m \in \mathbb{Z}^d \, \},
$$
denote the orthogonal lattice with mesh length $h$. We define
\begin{equation} \label{interior-domain}
\Omega_h =  \Omega \cap \mathbb{Z}^d_h \ \text{and} \  \partial \Omega_h =  \partial \Omega \cap \mathbb{Z}^d_h.
\end{equation}
We denote by $\mathcal{U}_h$ the linear space of mesh functions, i.e. real-valued functions defined on 
$$
\tir{\Omega}_h := \Omega_h \cup \partial \Omega_h.
$$
For $x \in \Omega_h$, $e \in \mathbb{Z}^d, e \neq 0$ such that $x \pm h e \in \tir{\Omega}_h$ and $v_h \in \mathcal{U}_h$, let
$$
\Delta_e v_h (x) = \frac{v_h(x+ h e ) - 2 v_h(x) + v_h(x- h e ) }{h^2 ||e||^2}.
$$

\begin{defn} \label{discrete-convex}
We say that a mesh function $v_h$ is {\it discrete convex} if and only if $\Delta_e v_h(x) \geq 0$ for all $x \in \Omega_h$ and $e \in \mathbb{Z}^d, e \neq 0$ such that $x \pm h e \in \tir{\Omega}_h$.
\end{defn}

Let us denote by $\mathcal{C}_h$ the cone of discrete convex mesh functions. If we define for $x \in \Omega_h$
$$
\lambda_{1,h}[v_h](x) = \min_{e \in \mathbb{Z}^d } \Delta_e v_h(x),
$$
then $v_h \in \mathcal{C}_h$ if and only if $\lambda_{1,h}[v_h] \geq 0$.

The distance of $x \in \R^d$ to a set $K$ is denoted $d(x,K)$. We make the usual convention of denoting constants by $C$ but will occasionally index some constants. For a function $w \in C(\Omega)$ 
its restriction on $\Omega_h$ is also denoted $w$ by an abuse of notation. Same for the restriction to $\partial \Omega_h$ of an element of $C(\partial \Omega)$.

For a set $K \subset \tir{\Omega}$ we denote by $\Conv(K)$ its convex hull. By the convexity of $\Omega$, $\Conv(\tir{\Omega}_h) \subset \tir{\Omega}$. Let $\mathcal{T}_h$ be a triangulation of $\Conv(\tir{\Omega}_h)$ with vertices in $\tir{\Omega}_h$ and denote by $I (v_h)$ the piecewise linear continuous function which is equal to $v_h$ on the set of vertices of $\tir{\Omega}_h$. The interpolant $I(v_h)$ is not necessarily a convex function. We make the following assumption on 
$\mathcal{T}_h$.

\begin{ass} \label{ass}
For a mesh function $v_h$, the interpolant $I(v_h)$ is piecewise linear along the coordinate axes, i.e. the line segments in $\tir{\Omega}$ through $x \in \Omega_h$ and directions $e \in \{ \, e_1,\ldots,e_d\, \}$ the canonical basis of $\R^d$.
\end{ass}

\begin{defn} \label{uni-def}
Let $v_h \in \mathcal{U}_h$ for each $h >0$. We say that $v_h$ converges to a function $v$ on $\tir{\Omega}$ uniformly on compact subsets of $\Omega$ if and only if $I(v_h)$  converges uniformly on compact subsets of $\Omega$ to $v$.

\end{defn}

We note that for each compact set $K \subset \Omega$, there exists $h_K$ such that $K \subset \Conv(\tir{\Omega}_h)$ for $h\leq h_K$. Thus $I(u_h)$ is defined on each compact set for $h$ sufficiently small. In particular, since points in $\R^d$ are compact, the same holds for $I(u_h)(x), x \in \Omega$.

The purpose of the introduction of the interpolant for the definition of uniform convergence on compact subsets is to prove that a bounded sequence of discrete convex functions is  locally equicontinuous, c.f. Lemma \ref{loc-equi} below. It is immediate that if $v_h$  converges uniformly on compact subsets of $\Omega$ to $v$ in the sense of Definition \ref{uni-def}, then for each compact set $K \subset \Omega$, each sequence $h_k \to 0$ and for all $\epsilon >0$, there exists $h_{-1} >0$ such that for all $h_k$, $0< h_k < h_{-1}$, we have
$$
\max_{x \in K \cap \mathbb{Z}_{h_k}^d} |v_{h_k}(x) - v(x)| < \epsilon.
$$

\subsection{Discrete Aleksandrov-Bakelman-Pucci's maximum principle}
We recall the results from \cite{awanou2019uweakcvg} on discrete versions of the notions of subdifferential. For $x \in \Omega$ we denote by $d(x,\partial \Omega)$ the distance of $x$ to $\partial \Omega$. For a subset $S$ of $\Omega$, $\diam(S)$ denotes its diameter.

For $v_h \in \mathcal{C}_h$ and $x_0 \in \Omega_h$, we define  
$$
\partial_h v_h (x_0) = \{ \, p \in \R^d: v_h(x) \geq v_h(x_0) + p \cdot (x-x_0), \, \text{for all} \, x \in \tir{\Omega}_h\,\},
$$
and
\begin{equation*}
M_h[v_h] (E) = | \partial_h v_h(E)  | \text{ for a Borel set } E.
\end{equation*}
It is shown in \cite{awanou2019uweakcvg} that the set function $M_h[v_h]$ defines a Borel measure for a discrete convex mesh function $v_h$ with the property 
$$
M_h[v_h] (\Omega) = \sum_{x \in \Omega_h} M_h[v_h] (\{ \, x \, \}). 
$$
We will need the following maximum principle.
\begin{lemma} \label{d-stab}
Let $v_h \in \mathcal{C}_h$ such that $v_h \geq 0$ on $\partial \Omega_h$. Then for 
$x \in \Omega_h$ 
$$
 v_h(x) \geq - C(d) \bigg[ \diam(\Omega)^{d-1} d(x,\partial \Omega) M_h[v_h](\Omega)\bigg]^{\frac{1}{d}},
$$
for a positive constant $C(d)$ which depends only on $d$. 
\end{lemma}

\begin{defn} \label{mass}
We refer to $M_h[v_h](\Omega)$ as  Monge-Amp\`ere mass of the discrete convex function $v_h$.
\end{defn}

\section{Dirichlet boundary condition for a uniform limit} \label{result}

The main result of this paper is proved in this section. We start with some lemmas.

\begin{lemma} \label{viscosity-sense}
Let $v_h \in \mathcal{C}_h$ denote a sequence of discrete convex functions which converges uniformly on compact subsets to a finite function $v$. Then the function $v$ is convex.
\end{lemma}

The proof of the above lemma is given in section \ref{proofs}.

The next lemma says that the extensions of 
bounded discrete convex functions are locally equicontinuous. For another example of a notion of discrete convexity for which equicontinuity holds, see \cite{Aguilera2008}. The continuous analogue of the next lemma can be found in \cite[Lemma 1.1.6]{Guti'errez2001}.

\begin{lemma} \label{loc-equi}
Assume that $v_h \in \mathcal{C}_h$ is  bounded. Then the family $I(v_h)$ is locally equicontinuous, i.e. for each compact subset $K \subset \Omega$,
there exist $h_K, C_K >0$ such that
$$
|I(v_h)(x) - I(v_h)(y)| \leq C_K ||x-y||, \, \forall x,y \in K, 
$$
and $h \leq h_K$.
\end{lemma}

\begin{proof} Recall that $I(v_h)$ is not necessarily a convex function. 
Let $(e_1, \ldots, e_d)$ denote the canonical basis of $\R^d$. For 
 $e \in \{ \,e_1, \ldots, e_d \, \}$, $x \in \tir{\Omega}_h$, $s \in \R$ such that $x+ s h e \in \tir{\Omega}_h$ we first show that for all $t \in [0,1]$
\begin{equation} \label{lin-conv}
I(v_h) (t x + (1-t) (x+ s h e)) \leq t I(v_h)(x) + (1-t) I(v_h)(x+ s h e).
\end{equation}
Since $v_h$ is discrete convex, we have $\Delta_e v_h (x) \geq 0$ for all $x \in \Omega_h$. Thus
$$
\frac{v_h(x+ h e ) - v_h(x)}{h}  \geq  \frac{v_h(x) - v_h(x- h e ) }{h}.
$$
The linear extension of $v_h$ on the segments connecting $x$ to $x+ h e$, $x \in \tir{\Omega}_h$,  which is thus convex, coincides with $I(v_h)$, by our choice of the interpolant $I(v_h)$, c.f. Assumption \ref{ass}. We conclude that \eqref{lin-conv} holds. 

Let $K$ be a compact subset of $\Omega$. We claim that for $r, s \in \R$ such that $x+r h e, x+ s h e \in K$ for some $x \in 
K \cap \Omega_h$, we have
\begin{equation} \label{lip-dir}
|I(v_h)(x+r h e) - I(v_h)(x+s  h e)| \leq C h |r-s|,
\end{equation} 
for a constant $C$ which depends on $K$ and $e_i$. The proof is similar to the one for \cite[Lemma 1.1.6]{Guti'errez2001}.

We denote by $\partial_e I(v_h)(y)$ the subdifferential of the restriction of $I(v_h)$ to the line segment through $x$ and direction $e$, i.e. $\partial_e I(v_h)(x+r h e) = \{ \, p \in \R, I(v_h)(x+r h e) \geq I(v_h)(x+t h e) + p h(r-t), t \in \R,  x + t he \in \Omega \, \}$. We then let $C= \sup \{ \, |p|, p \in \partial_e I(v_h)(K) \, \}$. We have $C \leq \max_{x \in K} |v_h(x+ h e ) - v_h(x)|/h $ and we recall that $||e||=1$. Moreover
$I(v_h)(x+r h e) \geq I(v_h)(x+s  h e) + p h (r-s)$ for $p \in \partial_e I(v_h)(x+s h e)$. Reversing the roles of $x+r h e$ and $x+t h e$ we obtain \eqref{lip-dir}.

We now prove \eqref{lip-dir} when $x \in K \cap (\Conv(\tir{\Omega}_h) \setminus \partial \Omega)$. 
There exists $h_K$ such that for $h \leq h_K$,  $K \subset \cup \{ \, T \in \mathcal{T}_h, T \cap \partial \Omega = \emptyset\, \}$. Thus for $x \in K$, we can find $\alpha_i, i=1,\ldots,d,  0 \leq \alpha_i \leq 1$ with $\sum_{i=1}^d \alpha_i =1$ such that
$x=\sum_{i=1}^d \alpha_i x_i$ and $I(v_h)$ is linear on the simplex $T$ with vertices $x_1, \ldots, x_d$ containing $x$. Moreover
$T \cap \partial \Omega = \emptyset$. We have
\begin{align*}
I(v_h)(x)  = \sum_{i=1}^d \alpha_i I(v_h)(x_i),
\end{align*}
and since $\sum_{i=1}^d \alpha_i =1$
\begin{align*}
I(v_h)(x+ r h e)  =\sum_{i=1}^d \alpha_i I(v_h)(x_i+ r h e) \text{ and }
I(v_h)(x+ s h e)  =\sum_{i=1}^d \alpha_i I(v_h)(x_i+ s h e).
\end{align*}
We thus have
$$
|I(v_h)(x+r h e_i) - I(v_h)(x+s h e_i)| \leq \sum_{i=1}^d \alpha_i |I(v_h)(x_i+r h e) - I(v_h)(x_i+s h e)| \leq C h |r-s|.
$$
As a consequence, $I(v_h)$ is Lipschitz continuous in each coordinate direction on compact subsets of $\Omega$. We conclude that $I(v_h)$ is locally Lipschitz continuous, hence locally equicontinuous.
\end{proof}

\begin{rem}
The proof of Lemma \ref{viscosity-sense} only exploits discrete convexity along the coordinate axes.
\end{rem}

Recall the discrete Laplacian
$$
\Delta_h v_h(x) =  \sum_{i=1}^d  \Delta_{e_i} v_h(x).
$$

For $x \in \tir{\Omega}$ define
\begin{equation} \label{bd-data}
U(x) = \sup \{ \, L(x), L \leq g \text{ on } \partial \Omega, L \text{ affine} \, \},
\end{equation}
the convex envelope with boundary data $g$. We have \cite[Theorem 5.2]{Hartenstine2006}

\begin{thm} \label{L-bd}
The function $g$ can be extended to a  convex function $\tilde{g} \in C(\tir{\Omega})$ if and only if the function $U$ defined by
\eqref{bd-data} is in $C(\tir{\Omega})$ and $U=g$ on $ \partial \Omega$.
\end{thm}

We can now state the main result of this paper.

\begin{thm} \label{main2}
Let $v_h \in \mathcal{C}_h$ be uniformly bounded with a uniform bound on their Monge-Amp\`ere masses, and such that $v_h=g$ on $\partial \Omega_h$. 
Assume that $g$ can be extended to a  convex function $\tilde{g} \in C(\tir{\Omega})$. Then, there is a subsequence $v_{h_k}$ which converges uniformly on compact subsets of $\Omega$ to a convex function $v$ in $C(\tir{\Omega})$  which solves $v=g$ on $\partial \Omega$. 
\end{thm}

\begin{proof} Since the family $v_h$ is uniformly bounded on
$\Omega_h$ we obtain by Lemma \ref{loc-equi} that $I(v_h)$ is locally equicontinuous.
By the Arzela-Ascoli theorem, there exists a subsequence $I(v_{h_k})$ which converges uniformly on compact subsets to a function 
$v$. 
More precisely, the subsequence $I(v_{h_k})$ is shown to be pointwise convergent on a dense subset of $\Omega$ with the convergence uniform on compact subsets. In particular, since points in $\R^d$ are compact, the convergence is also pointwise on $\Omega$.

Since $v_h \in \mathcal{C}_h$ the function $v$ is convex by Lemma \ref{viscosity-sense}. By the stability property, the function $v$ is locally bounded and hence continuous on $\Omega$. 

We now show that $v$ has a continuous extension to $\tir{\Omega}$, which we also denote by $v$ by an abuse of notation, and such that $v=g$ on $\partial \Omega$.

We first prove that for $\zeta \in \partial \Omega$, $\lim_{x \to \zeta} v(x) \geq g(\zeta)$ by arguing as in the proof of \cite[Lemma 5.1]{Hartenstine2006}. Let $\epsilon >0$. By Theorem \ref{L-bd} there exists an affine function $L$ such that $L \leq g$ on $\partial \Omega$ and $L(\zeta) \geq g(\zeta) - \epsilon$. Put $z_h=v_h-L$. Since $v_h=g$ on $\partial \Omega_h$, we have $z_h \geq 0$ on $\partial \Omega_h$. 

Now let $x \in \Omega$ and $x_h \in \Omega_h$ such that $x_h \to x$. Since $v_h$ converges to $v$ uniformly on compact subsets of $\Omega$, $z_h$ converges to $z$ uniformly on compact subsets of $\Omega$ and thus $z_h(x_h) \to v(x)-L(x) \equiv z(x)$. Assume that $z(x) <0$. 
By the discrete Aleksandrov's maximum principle Lemma \ref{d-stab}
 applied to $z_h$ on $\Omega$ we have
\begin{align*}
\begin{split} 
(-z_h(x_h))^d 
 & \leq C d(x_h, \partial \Omega) (\diam(\Omega))^{d-1}M_h[v_h](\Omega) \\
 & \leq C d(x_h, \partial \Omega) M_h[v_h](\Omega) \leq C ||x_h-\zeta|| M_h[v_h](\Omega).
 \end{split}
\end{align*}
By the assumption on the Monge-Amp\`ere masses $M_h[v_h](\Omega) \leq C$ with $C$ independent of $h$. Then
\begin{equation} \label{alex-applied}
(-z_h(x_h))^d \leq C ||x_h-\zeta||.
\end{equation}
Taking the limit as $h_k \to 0$ in \eqref{alex-applied}, we obtain for each $x \in \Omega$ for which $z(x) <0$
\begin{align*}
(-z(x))^d &\leq C ||x-\zeta||.
\end{align*}
In summary
$$
\text{either} \ z(x) \geq - C ||x-\zeta||^{\frac{1}{d}} \, \text{or} \, z(x) \geq 0, x \in \Omega.
$$
We conclude that
$$
v(x) \geq L(x) -C ||x-\zeta||^{\frac{1}{d}} \, \text{on} \, \Omega.
$$
Taking the limit as $x \to \zeta$ we obtain $\lim_{x \to \zeta} v(x) \geq g(\zeta)$.

Next, we prove that $\lim_{x \to \zeta} v(x) \leq g(\zeta)$. Since $v_h \in \mathcal{C}_h$, we have $\Delta_h v_h \geq 0$.  
Let $w_h$ denote the solution of the problem $\Delta_h w_h=0$ on $\Omega_h$ with $w_h=g$ on $\partial \Omega_h$. We have
$\Delta_h (v_h -w_h) \geq 0$ on $\Omega_h$ with $v_h-w_h=0$ on $\partial \Omega_h$. By the discrete maximum principle for the discrete Laplacian \cite[Theorem 4.77]{Hackbusch2010}, we have
$v_h-w^{}_h \leq 0$ on $\Omega_h$. 

Since a convex domain is Lipschitz, we can apply the results of \cite[section 6.2 ]{del2018convergence} and claim that 
 $w_h$ converges uniformly on compact subsets to the unique viscosity solution of $\Delta w =0$ on $\Omega$ with $w=g$ on
$\partial \Omega$. We then obtain $v(x) \leq w(x)$ on $\Omega$. But $w \in C(\tir{\Omega})$ \cite{del2018convergence}. We conclude that  $\lim_{x \to \zeta} v(x) \leq g(\zeta)$. Thus
$v \in C(\tir{\Omega})$ and $v=g$ on $\partial \Omega$.
\end{proof}

\section{Application to discretizations of the Monge-Amp\`ere equation} \label{cvg}

\subsection{Viscosity solutions of the elliptic Monge-Amp\`ere equation} \label{visc}

For given $f \geq 0$ continuous on $\tir{\Omega}$, we consider the Monge-Amp\`ere equation
\begin{align} \label{m1}
\begin{split}
\det D^2 u & = f \, \text{in} \, \Omega \\
u & = g \, \text{on} \, \partial \Omega.
\end{split}
\end{align}

A convex function $u \in C(\tir{\Omega})$ is a viscosity solution of \eqref{m1} if $u  = g \, \text{on} \, \partial \Omega$ and for all $\phi \in C^2(\Omega)$ the following holds
\begin{itemize}
\item[-] at each local maximum point $x_0$ of $u-\phi$, $f(x_0) \leq \det D^2 \phi(x_0)$
\item[-] at each local minimum point $x_0$ of $u-\phi$, $f(x_0) \geq \det D^2 \phi(x_0)$, if $D^2 \phi(x_0) \geq 0$, i.e. $D^2 \phi(x_0)$ has positive eigenvalues.
\end{itemize}

As explained in \cite{Ishii1990}, the requirement $D^2 \phi(x_0) \geq 0$ in the second condition above is natural for the two dimensional case we consider. 
The space of test functions in the definition above can be restricted to the space of strictly convex quadratic polynomials \cite[Remark 1.3.3]{Guti'errez2001}.

An upper semi-continuous convex function $u$ is said to be a viscosity sub solution of $\det D^2 u(x) = f(x)$ if the first condition holds and a lower semi-continuous convex function is said to be a viscosity super solution when the second holds. A  viscosity solution of \eqref{m1} is a continuous function which satisfies the boundary condition and is both a 
viscosity sub solution and a viscosity super solution.

Note that the notion of viscosity solution is a pointwise notion. It is not very difficult to prove that if $u$ 
is $C^2$ at $x_0$, then $u$ is a viscosity solution at the point $x_0$ of $\det D^2 u  = f $. 

For further reference, we recall the comparison principle of sub and super solutions, \cite[Theorem V. 2]{Ishii1990}. 
\begin{thm} \label{comparison}
Let $u$ and $v$ be respectively sub and super solutions of $\det D^2 u(x)=f(x)$ in $\Omega$.
Then if $\sup_{x \in \partial \Omega} \max(u(x) - v(x),0)=M$, then $u(x) - v(x) \leq M$ in $\Omega$.
\end{thm}

There are very few references which give an existence and uniqueness result for \eqref{m1} in the degenerate case $f \geq 0$. In \cite{Ishii1990} it is required that one can find a sub solution and a super solution. The difficulty is that the Monge-Amp\`ere equation is not often studied in convex but not necessarily strictly convex domains with the notion of viscosity solution. 

\subsection{Aleksandrov solutions and equivalence with viscosity solutions}

We recall that a convex function $u \in  C(\Omega)$  is an Aleksandrov solution of $\det D^2 u=f$ when its Monge-Amp\`ere measure is equal to the measure with density $f$. Aleksandrov solutions on convex but not necessarily strictly convex domains are studied in \cite{Hartenstine2006}.

We assume in addition that $f >0$.
We also assume that $g$ can be extended to a convex function $\tilde{g} \in C(\tir{\Omega})$. Then by \cite[Theorem 1.1]{Hartenstine2006}, \eqref{m1} has a unique Aleksandrov solution. The existence and uniqueness of a viscosity solution to \eqref{m1} in  $C(\tir{\Omega})$ then follows from the equivalence of viscosity and Aleksandrov solutions \cite[ Propositions 1.3.4 and 1.7.1]{Guti'errez2001}, under these assumptions. 
The equivalence of viscosity and Aleksandrov solutions in the degenerate case $f \geq 0$ is discussed in \cite{AwanouAwi22}.

\subsection{A reformulation of convexity} \label{ref-conv}
We recall that a function $\phi \in C^2(\Omega)$ is convex on $\Omega$ if the Hessian matrix $D^2 \phi$ is positive semidefinite or $\lambda_1[\phi] \geq 0$ where $\lambda_1[\phi]$ denotes the smallest eigenvalue of $D^2 \phi$. This notion was extended to continuous functions in \cite{Oberman07}. See also the remarks on  \cite[p. 226 ]{Trudinger97b}. 
An upper semi-continuous function $u$ is convex in the viscosity sense if and only if it is a viscosity solution of $-\lambda_1[u] \leq 0$, that is, for all $\phi \in C^2(\Omega)$, whenever $x_0$ is a local maximum point of $u-\phi$, $-\lambda_1[\phi] \leq 0$. 
This can also be written $\text{max}(-\lambda_1[u] ,0) = 0 \, \text{in} \, \Omega$, c.f. \cite{Oberman07}.

The Dirichlet problem for the Monge-Amp\`ere equation \eqref{m1} can then be written
\begin{align} \label{m11}
\begin{split}
-\det D^2 u + f & = 0 \, \text{in} \, \Omega \\
\text{max}(-\lambda_1[u] ,0) &= 0 \, \text{in} \, \Omega, \\
\end{split}
\end{align}
with boundary condition $u=g$ on $\partial \Omega$.
We write \eqref{m11} as $F(u)=0$ and 
note that the form of the equation is chosen to be consistent with the definition of ellipticity used for example in \cite{Ishii1990}.

Since we have now rewritten in \eqref{m11} convexity as an additional equation, sub solutions and super solutions of $-\det D^2 u + f  = 0$ do not need to be convex. We have the following comparison principle for \eqref{m11} \cite[Example 2.1 and Corollary 7.1]{BardiDragoni}: let $u^*$ be an upper semi-continuous sub solution of $-\det D^2 u + f  = 0$ which is convex in the viscosity sense and let $u_*$ be a lower semi-continuous super solution of $-\det D^2 u + f  = 0$ (which is not necessarily convex). Then
\begin{equation} \label{comparison-principle}
\sup_{\Omega} (u^*-u_*) \leq \max_{\partial \Omega} (u^*-v_*).
\end{equation}
A viscosity solution of \eqref{m11} is also a viscosity solution as defined in section \ref{visc}, since an upper semi-continuous function which is convex in the viscosity sense is also convex \cite[Example 2.1 and Theorem 3.1]{BardiDragoni}.

\subsection{Monotone schemes} \label{mono}

Let us denote by $F_h(v_h) \equiv \hat{F}_h(v_h(x), v_h(y)|_{ y \neq x})$ a discretization of $F(v)$. The scheme $F_h(v_h) =0$ is said to be monotone if for $z_h$ and $w_h$ in  $\mathcal{U}_h$, $z_h(y) \geq w_h(y), y \neq x$ implies $\hat{F}_h(z_h(x), z_h(y)|_{ y \neq x}) \geq \hat{F}_h(z_h(x), w_h(y)|_{ y \neq x})$. Here we use the partial ordering of $\R^d$, $(a_1,\ldots,a_d) \geq (b_1,\ldots,b_d)$ if and only if $a_i \geq b_i$ for all $i$.

The scheme is said to be consistent if for all $C^2$ functions $\phi$, and a sequence $x_h \to x \in \Omega$,
$\lim_{h \to 0} F_h (r_h(\phi)) (x_h) = F(\phi)(x)$. 

Finally the scheme is said to be stable if $F_h(v_h)=0$ has a solution $v_h$ which is bounded independently of $h$.

 Note that the convexity assumption on the exact solution is enforced through the definition of $F(v)$. In particular, if $F_h(v_h) =0$, then $v_h$ is discrete convex, i.e. $\lambda_{1,h}[v_h] \geq 0$ or equivalently  $\text{max}(-\lambda_{1,h}[v_h] ,0) = 0$. We consider the discretization

\begin{align} \label{m11mh2}
\begin{split}
F_h(u_h) & =0, \ \text{in} \ \Omega_h \\
u_h & = g\, \text{on} \, \partial \Omega_h. 
\end{split}
\end{align}

We make the assumption that the discretization is consistent, stable and monotone. In particular the half-relaxed limits
\begin{align*}
u^*(x) = \limsup_{y \to x, h \to 0} u_h(y) & = \lim_{\delta \to 0} \sup\{ \,u_h(y), y \in \Omega_h, |y-x| \leq \delta, 0<h\leq \delta \, \} \\
 u_*(x) = \liminf_{y \to x, h \to 0} u_h(y) & = \lim_{\delta \to 0} \inf \{ \,u_h(y), y \in \Omega_h, |y-x| \leq \delta, 0<h\leq \delta \, \},
\end{align*}
are then well defined.

It follows from \cite{Barles1991,nochetto2018two,DiscreteAlex2} that a consistent, stable and monotone scheme has a solution $v_h$.

\subsection{Convergence }

We close this section by stating the main application of the result of this paper.

\begin{thm} \label{singular} Let $f>0$ and $f\in C(\tir{\Omega})$. Assume that $g$ can be extended to a  convex function $\tilde{g} \in C(\tir{\Omega})$.
Let $F_h(u_h)=0$ be a  consistent, stable and monotone scheme for \eqref{m1} with $u_h$ solving \eqref{m11mh2} and
$M_h[u_h](\Omega) \leq C$ for a constant $C$ independent of $h$. 
Assume that the upper half-relaxed limit $u^*$ is a viscosity sub solution of $\det D^2 u(x)=f(x)$  and the lower half-relaxed limit $u_*$ is a viscosity super solution of $\det D^2 u(x)=f(x)$ in $\Omega$. Then
$u^* = g = u_*$ on $\partial \Omega$. Moreover, solutions $u_h$ of \eqref{m11mh2} converge uniformly on compact subsets to the unique viscosity solution of \eqref{m11}.
\end{thm}

\begin{proof} 
Since $u_h$ is uniformly bounded on $\Omega$ and discrete convex, by Theorem \ref{main2}, there exists a subsequence $u_{h_k}$ which converges uniformly on compact subsets of $\Omega$ to a convex function $v$ in $C(\tir{\Omega})$ which satisfies $v=g$ on $\partial \Omega$. 

It follows from the definitions that $v=u^*  = u_*$ on $\Omega$ and hence $v$ is a viscosity solution of $\det D^2 u=f$. 
By the comparison principle Theorem \ref{comparison}, $v$ is equal to the unique viscosity solution of \eqref{m1}. Thus all subsequences $u_{h_k}$ converge uniformly on compact subsets to the same limit. This concludes the proof.
\end{proof}

Several discrete Monge-Amp\`ere equations $F_h(u_h)=0$, e.g. \cite{Mirebeau14,Mirebeau15,DiscreteAlex2}, can be written
as
$$
\mathcal{M}_h[u_h](x) = h^d f(x), x \in \Omega_h, 
$$
for some operator $\mathcal{M}_h$ which satisfies
$$
M_h[u_h] \leq C \mathcal{M}_h[u_h].
$$
For $f \in L^1(\Omega), \sum_{x \in \Omega_h}  h^d f(x) \to \int_{\Omega} f(t) dt$ and thus 
$ \sum_{x \in \Omega_h} M_h[u_h](x) \leq  \sum_{x \in \Omega_h}  h^d f(x) \leq C$ for a constant $C$ independent of $h$, i.e. solutions of the discrete Monge-Amp\`ere equations have Monge-Amp\`ere masses uniformly bounded.

\section{Proof of Lemma \ref{viscosity-sense}} \label{proofs}

Since a function convex in the viscosity sense is convex, see for example \cite[Proposition 4.1]{Lindqvist00}, it is enough to show that the limit function $v$ is convex in the viscosity sense. We use the approach in \cite{Bouchard2009}.

By definition $I(v_h)$ is continuous on $\Omega$ and the convergence to $v$ is uniform on compact subsets of $\Omega$. Hence $v \in C(\Omega)$. 

Let $x_0 \in \Omega$ and $\phi \in C^2(\Omega)$ such that $v-\phi$ has a local maximum at $x_0$ with $(v-\phi)(x_0)=0$. Without loss of generality, we may assume that $x_0$ is a strict local maximum.

Let $B_0$ denote a closed ball contained in $\Omega$ and containing $x_0$ in its interior. We let $x_l$ be a sequence in $B_0 \cap \Omega_h$ such that $x_l \to x_0$ and $v_{h_l}(x_l) \to v(x_0)$ and let $x'_l$ be defined by
$$
c_l  \coloneqq  (v_{h_l} - \phi)(x'_l) = \max_{B_0\cap \Omega_h} (v_{h_l} - \phi).
$$
Since the sequence $x'_l$ is bounded, it converges to some $x_1$ after possibly passing to a subsequence. Since $(v_{h_l}-\phi)(x'_l) \geq (v_{h_l}-\phi)(x_l)$ we have
$$
(v-\phi)(x_0) = \lim_{l \to \infty} (v_{h_l}-\phi)(x_l) \leq \limsup_{l \to \infty} (v_{h_{l}}-\phi)(x_l')=  \limsup_{l \to \infty} c_l \leq (v-\phi)(x_1).
$$
Since $x_0$ is a strict maximizer of the difference $v-\phi$, we conclude that $x_0=x_1$ and $c_l \to 0$ as $l \to \infty$.

By definition
\begin{equation} \label{order}
v_{h_l}(x) \leq \phi(x) + c_l, \forall x \in B_0.
\end{equation}

We recall that for $v_h \in \mathcal{C}_h$, $-\lambda_{1,h}[v_h] \leq 0$. Now, the operator $\lambda_{1,h}[v_h]$ is easily seen to be monotone. In addition it is consistent. Put $F_h[v_h](x_0)=\lambda_{1,h}[v_h](x_0)=\hat{F}_{h}(v_{h}(x_0), v_{h}(y)|_{y \neq x_0})$. 

By the monotonicity of the scheme we obtain from \eqref{order}
$$
0 \leq \hat{F}_{h_l}(v_{h_l}(x_0), v_{h_l}(y)|_{y \neq x_0}) \leq \hat{F}_{h_l}(v_{h_l}(x_0), (\phi(y) +c_l)|_{y \neq x_0}), 
$$
which gives by the consistency of the scheme $\lambda_1[\phi](x_0) \geq 0$. This concludes the proof.


\section*{Acknowledgments}

The author was partially supported by NSF grants DMS-1319640 and  DMS-1720276. The author is grateful to Michael Neilan for pointing out the work of reference \cite{nochetto2018two} on the Dirichlet boundary condition of a uniform limit. The author would like to thank the Isaac Newton Institute for Mathematical Sciences, Cambridge, for support and hospitality during the programme ''Geometry, compatibility and structure preservation in computational differential equations'' where part of this work was undertaken. Part of this work was supported by EPSRC grant no EP/K032208/1.






\begin{thebibliography}{10}

\bibitem{Aguilera2008}
{\sc N.~E. Aguilera and P.~Morin}, {\em Approximating optimization problems
  over convex functions}, Numer. Math., 111 (2008), pp.~1--34.

\bibitem{DiscreteAlex2}
{\sc G.~Awanou}, {\em Discrete {A}leksandrov solutions of the
  {M}onge-{A}mp\`ere equation}.
\newblock { h}ttps://arxiv.org/abs/1408.1729.

\bibitem{awanou2019uweakcvg}
{\sc G.~Awanou}, {\em Weak convergence of {M}onge-{A}mp\`ere measures for
  discrete convex mesh functions},  (2019).
\newblock https://arxiv.org/abs/1910.13870.

\bibitem{AwanouAwi22}
{\sc G.~Awanou and R.~Awi}, {\em Convergence of finite difference schemes to
  the {A}leksandrov solution of the {M}onge-{A}mp{\`e}re equation}, Acta
  Applicandae Mathematicae, 144 (2016), pp.~87--98.

\bibitem{Baginski96}
{\sc F.~E. Baginski and N.~Whitaker}, {\em Numerical solutions of boundary
  value problems for $k$-surfaces in $\mathbb{R}^3$}, Numer. Methods Partial
  Differential Equations, 12 (1996), pp.~525--546.

\bibitem{BardiDragoni}
{\sc M.~Bardi and F.~Dragoni}, {\em Convexity and semiconvexity along vector
  fields}, Calc. Var. Partial Differential Equations, 42 (2011), pp.~405--427.

\bibitem{Barles1991}
{\sc G.~Barles and P.~E. Souganidis}, {\em Convergence of approximation schemes
  for fully nonlinear second order equations}, Asymptotic Anal., 4 (1991),
  pp.~271--283.

\bibitem{Mirebeau14}
{\sc J.-D. Benamou, F.~Collino, and J.-M. Mirebeau}, {\em Monotone and
  consistent discretization of the {M}onge-{A}mp\`ere operator}, Math. Comp.,
  85 (2016), pp.~2743--2775.

\bibitem{Bohmer2008}
{\sc K.~B{\"o}hmer}, {\em On finite element methods for fully nonlinear
  elliptic equations of second order}, SIAM J. Numer. Anal., 46 (2008),
  pp.~1212--1249.

\bibitem{Bouchard2009}
{\sc B.~Bouchard, R.~Elie, and N.~Touzi}, {\em Discrete-time approximation of
  {BSDE}s and probabilistic schemes for fully nonlinear {PDE}s}, in Advanced
  financial modelling, vol.~8 of Radon Ser. Comput. Appl. Math., Walter de
  Gruyter, Berlin, 2009, pp.~91--124.

\bibitem{Belgacem2006}
{\sc M.~Bouchiba and F.~B. Belgacem}, {\em Numerical solution of
  {M}onge-{A}mp\`ere equation}, Math. Balkanica (N.S.), 20 (2006),
  pp.~369--378.

\bibitem{Brenner2010b}
{\sc S.~C. Brenner, T.~Gudi, M.~Neilan, and L.-Y. Sung}, {\em {$C^0$} penalty
  methods for the fully nonlinear {M}onge-{A}mp\`ere equation}, Math. Comp., 80
  (2011), pp.~1979--1995.

\bibitem{Brix}
{\sc K.~Brix, Y.~Hafizogullari, and A.~Platen}, {\em Solving the
  {M}onge--{A}mp\`ere equations for the inverse reflector problem}, Math.
  Models Methods Appl. Sci., 25 (2015), pp.~803--837.

\bibitem{Glowinski2014}
{\sc A.~Caboussat, R.~Glowinski, and D.~C. Sorensen}, {\em A least-squares
  method for the numerical solution of the {D}irichlet problem for the elliptic
  {M}onge-{A}mp\`ere equation in dimension two}, ESAIM Control Optim. Calc.
  Var., 19 (2013), pp.~780--810.

\bibitem{Chen2010b}
{\sc Y.~Chen and S.~R. Fulton}, {\em An adaptive continuation-multigrid method
  for the balanced vortex model}, J. Comput. Phys., 229 (2010), pp.~2236--2248.

\bibitem{daniel2014quelques}
{\sc J.-P. Daniel}, {\em Quelques r{\'e}sultats d'approximation et de
  r{\'e}gularit{\'e} pour des {\'e}quations elliptiques et paraboliques
  non-lin{\'e}aires}, PhD thesis, Paris 6, 2014.

\bibitem{Davydov12}
{\sc O.~Davydov and A.~Saeed}, {\em Numerical solution of fully nonlinear
  elliptic equations by {B}\"ohmer's method}, J. Comput. Appl. Math., 254
  (2013), pp.~43--54.

\bibitem{del2018convergence}
{\sc F.~del Teso, J.~J. Manfredi, and M.~Parviainen}, {\em Convergence of
  dynamic programming principles for the $ p $-laplacian}, arXiv preprint
  arXiv:1808.10154,  (2018).

\bibitem{Delzanno08}
{\sc G.~L. Delzanno, L.~Chac{\'o}n, J.~M. Finn, Y.~Chung, and G.~Lapenta}, {\em
  An optimal robust equidistribution method for two-dimensional grid adaptation
  based on {M}onge-{K}antorovich optimization}, J. Comput. Phys., 227 (2008),
  pp.~9841--9864.

\bibitem{Feng2017}
{\sc X.~Feng and M.~Jensen}, {\em Convergent semi-{L}agrangian methods for the
  {M}onge-{A}mp\`ere equation on unstructured grids}, SIAM J. Numer. Anal., 55
  (2017), pp.~691--712.

\bibitem{FengLewis14}
{\sc X.~Feng and T.~Lewis}, {\em Mixed interior penalty discontinuous
  {G}alerkin methods for fully nonlinear second order elliptic and parabolic
  equations in high dimensions}, Numer. Methods Partial Differential Equations,
  30 (2014), pp.~1538--1557.

\bibitem{Feng2009}
{\sc X.~Feng and M.~Neilan}, {\em Analysis of {G}alerkin methods for the fully
  nonlinear {M}onge-{A}mp\`ere equation}, J. Sci. Comput., 47 (2011),
  pp.~303--327.

\bibitem{Oberman2010a}
{\sc B.~Froese and A.~Oberman}, {\em Convergent finite difference solvers for
  viscosity solutions of the elliptic {M}onge-{A}mp\`ere equation in dimensions
  two and higher}, SIAM J. Numer. Anal., 49 (2011), pp.~1692--1714.

\bibitem{GlowinskiICIAM07}
{\sc R.~Glowinski}, {\em Numerical methods for fully nonlinear elliptic
  equations}, in I{CIAM} 07---6th {I}nternational {C}ongress on {I}ndustrial
  and {A}pplied {M}athematics, Eur. Math. Soc., Z\"urich, 2009, pp.~155--192.

\bibitem{Guti'errez2001}
{\sc C.~E. Guti{\'e}rrez}, {\em The {M}onge-{A}mp\`ere equation}, Progress in
  Nonlinear Differential Equations and their Applications, 44, Birkh\"auser
  Boston Inc., Boston, MA, 2001.

\bibitem{Hackbusch2010}
{\sc W.~Hackbusch}, {\em Elliptic differential equations}, vol.~18 of Springer
  Series in Computational Mathematics, Springer-Verlag, Berlin, english~ed.,
  2010.
\newblock Theory and numerical treatment, Translated from the 1986 corrected
  German edition by Regine Fadiman and Patrick D. F. Ion.

\bibitem{Hartenstine2006}
{\sc D.~Hartenstine}, {\em The {D}irichlet problem for the {M}onge-{A}mp\`ere
  equation in convex (but not strictly convex) domains}, Electron. J.
  Differential Equations,  (2006), pp.~No. 138, 9 pp. (electronic).

\bibitem{Ishii1990}
{\sc H.~Ishii and P.-L. Lions}, {\em Viscosity solutions of fully nonlinear
  second-order elliptic partial differential equations}, J. Differential
  Equations, 83 (1990), pp.~26--78.

\bibitem{jensen2017numerical}
{\sc M.~Jensen}, {\em Numerical solution of the simple {M}onge-{A}mp\`ere
  equation with non-convex {D}irichlet data on non-convex domains}, arXiv
  preprint arXiv:1705.04653,  (2017).

\bibitem{Lakkis11b}
{\sc O.~Lakkis and T.~Pryer}, {\em A finite element method for nonlinear
  elliptic problems}, SIAM J. Sci. Comput., 35 (2013), pp.~A2025--A2045.

\bibitem{Lindqvist00}
{\sc P.~Lindqvist, J.~Manfredi, and E.~Saksman}, {\em Superharmonicity of
  nonlinear ground states}, Rev. Mat. Iberoamericana, 16 (2000), pp.~17--28.

\bibitem{Mirebeau15}
{\sc J.-M. Mirebeau}, {\em Discretization of the 3{D} {M}onge-{A}mpere
  operator, between wide stencils and power diagrams}, ESAIM Math. Model.
  Numer. Anal., 49 (2015), pp.~1511--1523.

\bibitem{NeilanSalgadoZhang2017}
{\sc M.~Neilan, A.~J. Salgado, and W.~Zhang}, {\em Numerical analysis of
  strongly nonlinear {PDE}s}, Acta Numer., 26 (2017), pp.~137--303.

\bibitem{neilan2019monge}
{\sc M.~Neilan, A.~J. Salgado, and W.~Zhang}, {\em The {M}onge-{A}mp\`ere
  equation}, arXiv preprint arXiv:1901.05108,  (2019).

\bibitem{nochetto2018two}
{\sc R.~H. Nochetto, D.~Ntogkas, and W.~Zhang}, {\em Two-scale method for the
  {M}onge-{A}mp\`ere equation: convergence to the viscosity solution}, Math.
  Comp., 88 (2019), pp.~637--664.

\bibitem{Oberman07}
{\sc A.~M. Oberman}, {\em The convex envelope is the solution of a nonlinear
  obstacle problem}, Proc. Amer. Math. Soc., 135 (2007), pp.~1689--1694
  (electronic).

\bibitem{Sulman11}
{\sc M.~M. Sulman, J.~F. Williams, and R.~D. Russell}, {\em An efficient
  approach for the numerical solution of the {M}onge-{A}mp\`ere equation},
  Appl. Numer. Math., 61 (2011), pp.~298--307.

\bibitem{Trudinger97b}
{\sc N.~S. Trudinger and X.-J. Wang}, {\em Hessian measures. {I}}, Topol.
  Methods Nonlinear Anal., 10 (1997), pp.~225--239.
\newblock Dedicated to Olga Ladyzhenskaya.

\bibitem{Zheligovsky10}
{\sc V.~Zheligovsky, O.~Podvigina, and U.~Frisch}, {\em The {M}onge-{A}mp\`ere
  equation: various forms and numerical solution}, J. Comput. Phys., 229
  (2010), pp.~5043--5061.

\end{thebibliography}

\end{document}